\documentclass[12pt]{article}
\usepackage[english]{babel}
\usepackage{inputenc}
\usepackage[sumlimits]{amsmath}
\usepackage{amsthm}
\usepackage{amstext}
\usepackage{amssymb}
\usepackage{mathrsfs}

\newtheorem{lemma}{Lemma}
\newtheorem{proposition}{Proposition}
\newtheorem{corollary}{Corollary}

\DeclareMathOperator{\ara}{ara}
\DeclareMathOperator{\pd}{pd}

\title{The arithmetical rank of a special class of monomial ideals}
\author{Pietro Mongelli}

\begin{document}
%\setcounter{page}{1}
%\pagenumbering{arabic}

\maketitle
\begin{center}{e-mail address: mongelli@mat.uniroma1.it}\end{center}
{\bf AMS 2010:} 13A15 (Primary); 13F55 (Secondary).\\\\\\
%\tableofcontents
Given a commutative noetherian ring with identity $R$, the arithmetical rank of an ideal $I$ of $R$, denoted by $\ara I$, is defined as the minimum number of elements that generate an ideal having the same radical as $I$. When $I$ is a square-free monomial ideal, then Lyubeznik \cite{Lyubeznik} proved that
\begin{equation}
\label{E:dis}
\pd(R/I)\le \ara I,
\end{equation}
where $\pd(R/I)$ denotes the projective dimension of the module $R/I$. 
In this paper we want to generalize result in \cite{Jing}, Theorem 4.5, verifing that for all $t,n$ positive integers the equality
\begin{equation}
\label{E:fond}
\ara(I_t(L_n))=\pd(I_t(L_n))
\end{equation}
holds, where
$$
I_t(L_n)=\{x_ix_{i+1}\cdots x_{i+t-1}\vert i=1,\dots, n-t+1\}.
$$

To prove \ref{E:fond}, we need the following lemma.

\begin{lemma}
\label{L:IncrementoVariabili}
Let $I, J$ be monomial ideals in the ring  $K[x_1,\dots,x_N]$. Suppose that each generator of $J$ is divisible by an indeterminate $x_i$ (not necessary the same for all monomials) such that no generators of $I$ is multiple of $x_i$. Then $\ara I\le \ara (I+ J)$.
\end{lemma}

\begin{proof}
Without loss of generality, we can assume that all monomials generating $I,J$ are square-free, thus $I,J$ are radical ideals.
We can also assume that $I$ is contained in $K[x_1,\dots,x_r]$ for some $r$ and that all monomials of $J$ are divisible by some $x_j$ with $j>r$.
Let $n=\ara (I+ J)$ and let $g_1,\dots,g_n$ be elements such that $\sqrt{(g_1,\dots,g_n)}=(I+ J)$. For all $i=1,\dots, n$ we set $f_i$ equals to $g_i$ without monomials divisible by $x_j$ for some $j>r$. We now show that $\sqrt{(f_1,\dots,f_n)}=I$.\\
Since $I,J$ are monomial ideals, obviously $f_i\in I$ for all $i\le n$. Let $h$ be an element of $I$, then $h\in I+ J$, so $h^m\in (g_1,\dots,g_n)$ for some integer $m$. Then, $h^m=\displaystyle\sum_{i=1}^n a_ig_i$ with  $a_i\in K[x_1,\dots,x_N]$. We know that $h\in K[x_1,\dots,x_r]$ so every linear combination with indeterminates $x_j$, $j>r$ must be $0$. Then $h^m=\displaystyle\sum_{i=1}^n b_if_i$ where each $b_i$ is obtained by $a_i$, deleting some monomials. Then $h\in \sqrt{(f_1,\dots,f_n)}$.
\end {proof}

Note that in general inequality $\ara I\le \ara (I+ J)$ is not true. In fact, let $I=(x_1x_2,x_1x_3,x_4x_5), J=(x_1)$ in $K[x_1,\dots,x_5]$. $I$ has $(x_2,x_3,x_4)$ as minimal prime ideal; this prime ideal has height $3$ and is known that this is a lower bound for the arithmetical rank of $I$. So $\ara I=3$, but $\ara (I+ J)=2$ since $I+ J=(x_1,x_4x_5)$.

We now present a result due to Barile \cite{Barile1996}.
\begin{proposition}
\label{P:caso semplice}
For all $n\ge 1$ let $I_n$ be the ideal of $K[x_1,\dots,x_{2n}]$ generated by the following monomials
\begin{equation*}
A=\{x_1\cdots x_{n}, x_2\cdots x_{n+1},\dots, x_{n+1}\cdots x_{2n}\}.
\end{equation*}
Then, $\ara A=2$.
\end{proposition}
\begin{proof}
\cite {Barile1996}, Proposition 3.1.
\end{proof}

Now we can give a proof of (\ref{E:fond}).

\begin{corollary}
\label{C:path}
%va nelle premesse
\begin{equation*}
 \ara(R/I_t(L_n))=\left\{\begin{array}{ll}
 \frac{2(n-d)}{t+1} & \textrm{if } n\equiv d \textrm{ mod } (t+1) \text{ with } 0\le d\le t-1\\
 \frac{2n-(t-1)}{t+1} & \textrm{if } n\equiv t \textrm{ mod } (t+1).
 \end{array}\right.
 \end{equation*}
\end{corollary}

\begin {proof}
We can write $n=k(t+1)+d$ with $0\le d\le t$. We know that $I_t(L_n)$ consists of $n-t+1=k(t+1)+d-t+1$ monomials of the form $x_i\dots x_{i+t-1}$.
By previous proposition, the radical of the ideal  is the same if we substitute $t+1$ consecutive monomial by two opportune polynomials. For this reason, we subdivide these monomials in sets of $t+1$ elements.
First we consider the case $d=t-1$. Then we get exactly $k$ sets, so by previous proposition $\ara I_t(L_n)\le 2k=\frac{2(n-d)}{t+1}$. \\
Now suppose $d<t-1$. Let $J$  be the ideal generated by monomials $x_i\cdots x_{i+t-1}$ for $i=k(t+1)+d-t+2,\dots,k(t+1)$. The ideal $I_t(L_n)+J$ is in the case $d=t-1$ so $\ara (I_t(L_n)+J)\le 2k$. By Lemma \ref{L:IncrementoVariabili} we get $\ara I_t(L_n)\le 2k=\frac{2(n-d)}{t+1}$.\\
Finally suppose $d=t$, then there are $k(t+1)+1$ monomials. As in the previous cases, we can replace the first $k(t+1)$ elements with $2k$ polynomials, without change the radical of the ideal, so we get $2k+1$ elements. Therefore $\ara I_t(L_n)\le 2k+1=\frac{2n-(t-1)}{t+1}$. 
 
%Aggiungiamo all'ultimo gruppo tanti monomi in nuove indeterminate affinchè otteniamo la cardinalità $t+1$. Per il lemma, questa operazione non aumenta in rango aritmetico. Effettuando le sostituzioni otteniamo un numero di generatori a meno di radicale pari a 2 volte il numero di tali insiemi. Esplicitamente tale numero è l'intero arrotondato per eccesso del rapporto $\frac{n-t+1}{t+1}$. Esso vale $2k$ se $d-t+1\le 0$, ossia per $d\neq t$. In tal caso $2k=\frac{2(n-d}{t+1}$. Nel caso $d=t$, l'ultimo gruppo è costituito da un solo elemento: esso, insieme ai $2k$ elementi dei precedenti insiemi, genera a meno di radicale $I_t(L_n)$. Ma $\frac{2n-(t-1)}{t+1}=2k+1$ e ciò basta per concludere le stime superiori.
The claim follows immediately by projective dimension of $R/I_t(L_n)$ computated in \cite{Jing}, Theorem 4.1 and by \ref{E:dis}.
\end{proof}

\end{document}